\documentclass[a4paper,10pt,reqno]{amsart}
\usepackage{amsmath,amsfonts,amssymb,amsthm}
\usepackage{mathtools}
\usepackage{mathrsfs}
\usepackage{color}
\usepackage[margin=3cm]{geometry}
\usepackage[colorlinks=true,citecolor=blue,urlcolor=blue]{hyperref}
\usepackage{caption,enumerate}

\newtheorem{theorem}{Theorem}[section]
\newtheorem{lemma}[theorem]{Lemma}

\newtheorem{corollary}[theorem]{Corollary}
\newtheorem{remark}{Remark}[section]
\numberwithin{equation}{section}

\newcommand{\CC}{\ensuremath{\mathbb{C}}}
\newcommand{\RR}{\ensuremath{\mathbb{R}}}
\newcommand{\ZZ}{\ensuremath{\mathbb{Z}}}

\newcommand{\md}{\mathrm{d}}
\newcommand{\ve}{\varepsilon}
\newcommand{\mo}{\mathcal{O}}
\newcommand{\sm}{\setminus}

\usepackage{float}
\usepackage{indentfirst}
\usepackage{changepage}
\usepackage{setspace}
\usepackage{graphicx}
\usepackage{calc}

\usepackage{bbm}

\allowdisplaybreaks

\author[P.-C. Hang]{Peng-Cheng Hang}
\address{Peng-Cheng Hang\\Department of Mathematics, School of Mathematics and Statistics, Donghua University, Shanghai 201620, People's Republic of China.}\email{mathroc618@outlook.com}

\author[M.-J. Luo]{Min-Jie Luo}
\address{Min-Jie Luo\\Department of Mathematics, School of Mathematics and Statistics, Donghua University, Shanghai 201620,
	People's Republic of China.}\email{mathwinnie@live.com; mathwinnie@dhu.edu.cn}

\keywords{Riemann zeta function, $a$-points, value-distribution}
\subjclass[2020]{11M06, 11M26, 41A60}

\begin{document}
\title[Note on the $a$-points of $\zeta$-function]{Note on the $a$-points of the Riemann zeta function}
	
\begin{abstract}
	For any $a\in\CC$, the zeros of $\zeta(s)-a$, denoted by $\rho_a=\beta_a+i\gamma_a$, are called $a$-points of the Riemann zeta function $\zeta(s)$. In this paper, we reformulate some basic results about the $a$-points of $\zeta(s)$ shown by Garunk\v{s}tis and Steuding. We then deduce an asymptotic of the sum
	\[S_T(a,\delta)=\sum_{\tau<\gamma_a\leqslant T}\zeta'(\rho_a+i\delta)X^{\rho_a},\quad T\to\infty,\]
	where $0\ne\delta=\frac{2\pi\alpha}{\log\frac{T}{2\pi X}}\ll 1$, and $X>0$ and $\tau\geqslant|\delta|+1$ are fixed.
	
	We also find the interesting varied behavior of $S_T(a,\delta)$ in different $X$ ranges, which is more complicated than those described before by Gonek and Pearce-Crump.
\end{abstract}
	
\maketitle

\section{\bf Introduction}
	Let $\zeta(s)$ be the Riemann zeta function, $s=\sigma+it$ be a complex variable, $\rho=\beta+i\gamma$ be a non-trivial zero of $\zeta(s)$ and $a$ be any complex number. The zeros of $\zeta(s)-a$, say $\rho_a=\beta_a+i\gamma_a$, are called $a$-points of $\zeta(s)$. Additionally, $a$-points with $\beta_a\leqslant 0$ are said to be trivial. Clearly, when $a=0$, the $a$-points reduce to the ordinary zeros.
	
	There are many interesting investigations on the distribution of the values of the derivatives $\zeta^{(n)}(s)$ of $\zeta(s)$ at its zeros. For instance, Fujii proved in \cite{Fuji--2} that
	\begin{equation}\label{Fujii's result--0}
		\sum_{0<\gamma\leqslant T}\zeta'(\rho)=\frac{T}{4\pi}\log^2\left(\frac{T}{2\pi}\right)+(C_0-1)\frac{T}{2\pi}\log\frac{T}{2\pi}+\left(1-C_0-C_0^2+3C_1\right)\frac{T}{2\pi}+E(T),
	\end{equation}
	where $E(T)$ is given explicitly both unconditionally and assuming Riemann Hypothesis, and $C_0$ and $C_1$ are the constants in the expansion
	\[\zeta(s)=\frac{1}{s-1}+C_0+C_1(s-1)+\cdots.\]
	But there are lesser studies on the distribution concerning the $a$-points.
	
	A fundamental result on the $a$-points is the Riemann-von Mangoldt-type formula. In 1913, Landau \cite{BoLL} estimated the number of nontrivial $a$-points, namely
	\begin{equation}\label{Riemann-von Mangoldt formula}
		N_a(T):=\sum_{\substack{\beta_a>0\\1<\gamma_a\leqslant T}}1=\frac{T}{2\pi}\log\frac{T}{2\pi ec_a}+\mo(\log T),\quad T\to\infty
	\end{equation}
	with the constant $c_a=1$ if $a\ne 1$, and $c_1=2$. As with the usual Riemann-von Mangoldt formula, \eqref{Riemann-von Mangoldt formula} is of great importance in studying the distribution of $a$-points.
	
	In 2014, Garunk\v{s}tis and Steuding \cite{GaSt} showed that there are countably many trivial $a$-points, but their proof is not fully rigorous. They generalized \eqref{Fujii's result--0} by using \eqref{Riemann-von Mangoldt formula} and proving that
	\begin{equation}\label{GaSt's result}
		\sum_{0<\gamma_a\leqslant T}\zeta'(\rho_a)=\sum_{0<\gamma\leqslant T}\zeta'(\rho)-\frac{aT}{2\pi}\left\{\log^2\left(\frac{T}{2\pi}\right)-2\log\frac{T}{2\pi}+2\right\}+\mo\left(T^{\frac{1}{2}+\ve}\right).
	\end{equation}
	They also showed
	\begin{align}
		\sum_{0<\gamma_a\leqslant T}{}& \left(\zeta(\rho_a+i\delta)-a\right) \nonumber\\
		={}& \left(1-\frac{\sin(2\pi\alpha)}{2\pi\alpha}+i\pi\alpha\left(\frac{\sin(\pi\alpha)}{\pi\alpha}\right)^2-a(1-\cos(2\pi\alpha)+i\sin(2\pi\alpha))\right)\frac{T}{2\pi}\log\frac{T}{2\pi} \nonumber\\
		& +\frac{T}{2\pi}\Biggl(-1+e^{-2\pi i\alpha}\left(\frac{1}{i\delta}\left(\frac{1}{1-i\delta}-1\right)-\frac{1}{1-i\delta}\left(\zeta(1-i\delta)+\frac{1}{i\delta}\right)\right) \nonumber\\
		& +\frac{\zeta'}{\zeta}(1+i\delta)+\frac{1}{i\delta}-a\,\biggl\{1+\log c_a+\bigl(\cos(2\pi\alpha)-i\sin(2\pi\alpha)\bigr)\times \nonumber\\
		& \times\biggl(\frac{1}{\left(1-i\delta\right)^2}-\frac{2\pi i\alpha}{1-i\delta}\biggr)\biggr\}\Biggr)+E(T) \label{GaSt's result--1}
	\end{align}
	uniformly in $\alpha$, where $0\ne \delta=\frac{2\pi\alpha}{\log\frac{T}{2\pi}}\ll 1$. They further remarked in \cite[p. 13]{GaSt} that differentiation of the above formula with respect to $\alpha$ leads to \eqref{GaSt's result}. The idea of differentiation seems to originate in the work of Ingham in \cite[p. 275]{Ingh}.
	
	Jakhlouti and Mazhouda \cite{JaMa} then generalized \eqref{GaSt's result} by showing that for any fixed $X>0$,
	\begin{equation}\label{JaMa's result}
		\sum_{0<\gamma_a\leqslant T}\zeta'(\rho_a)X^{\rho_a}=\sum_{0<\gamma\leqslant T}\zeta'(\rho)X^{\rho}-\frac{aT}{2\pi}\left\{\log^2\left(\frac{T}{2\pi}\right)-2\log\frac{T}{2\pi}+2\right\}+\mo\left(T^{\frac{1}{2}+\ve}\right),
	\end{equation}
	where the sum on the right was evaluated by Fujii \cite{Fuji--2}\,: for any fixed $X>0$,
	\begin{equation}\label{Fujii's result}
		\begin{split}
			\sum_{0<\gamma\leqslant T}\zeta'(\rho)X^{\rho}={}& -\Delta(X)\frac{T}{2\pi}\left\{\left(\frac{1}{2}\log\frac{T}{2\pi}-\frac{1}{2}+\frac{\pi i}{4}\right)\log X-\sum_{mr=X}\Lambda(r)\log m\right\}\\
			& +X\sum_{m\leqslant\frac{T}{2\pi X}}e^{2\pi imX}\log^2 m+\frac{1}{2}X\log X\sum_{m\leqslant\frac{T}{2\pi X}}e^{2\pi imX}\log m\\
			& -\frac{1}{4}X\log X\left(2\log X-\pi i\right)\sum_{m\leqslant\frac{T}{2\pi X}}e^{2\pi imX}-X\sum_{mr\leqslant\frac{T}{2\pi X}}e^{2\pi imrX}\Lambda(r)\log m\\
			& +\mo\left(T^{\frac{1}{2}}\log^3 T\right)
		\end{split}
	\end{equation}
	with
	\begin{equation}\label{Delta(X) definition}
		\Delta(X)=\begin{cases}
			1,& \text{if }X\in\ZZ\\
			0,& \text{if }X\notin\ZZ
		\end{cases}
	\end{equation}
	and $\Lambda(r)$ denoting the von Mangoldt function defined by
	\[\Lambda(r)=\begin{cases}
		\log p,& \text{if }r=p^k\text{ for some prime }p\text{ and some integer }k\geqslant 1,\\
		0,& \text{otherwise}.
	\end{cases}\]
	However, there is a minor mistake in \eqref{JaMa's result}.
	
	In this paper, we shall revisit Garunk\v{s}tis and Steuding's work on the trivial $a$-points and also make a revision of Jakhlouti and Mazhouda's result \eqref{JaMa's result}. Motivated by considering \eqref{GaSt's result--1}, we shall generalize \eqref{JaMa's result} by deriving the asymptotics of
	\begin{equation}\label{our sum}
		S_T(a,\delta):=\sum_{\tau<\gamma_a\leqslant T}\zeta'(\rho_a+i\delta)X^{\rho_a},\quad T\to\infty,
	\end{equation}
	where $X>0$ and $\tau\geqslant|\delta|+1$ are fixed, and $0\ne\delta=\frac{2\pi\alpha}{\log\frac{T}{2\pi X}}\ll 1$.
	
	The advantage of considering \eqref{our sum} is that taking $a=0$ and differentiating it with respect to $\alpha$ can yield the work of Pearce-Crump \cite{PeCr} on the asymptotics of
	\[\sum_{0<\gamma\leqslant T}\zeta^{(n)}(\rho)X^{\rho},\quad T\to\infty.\]
	In particular, setting $X=1$ in the corresponding asymptotic expansion recovers \cite[Theorem 3]{HuPC}.

\section{\bf Statement of results}
    Throughout this paper, $a\in\CC$ and $X>0$ are assumed to be any fixed numbers. Let $\rho=\beta+i\gamma$ be a non-trivial zero of $\zeta(s)$ and let $\rho_a=\beta_a+i\gamma_a$ be any $a$-point of $\zeta(s)$. Write $s=\sigma+it$ with $\sigma,t\in\RR$, use the notation for the characteristic function
    \[\mathbbm{1}_{x\in \Omega}:=\begin{cases}
    	1,& \text{if }x\in\Omega\\
    	0,& \text{otherwise}
    \end{cases}\]
    and recall the function $\Delta(X)=\mathbbm{1}_{X\in\ZZ}$ defined in \eqref{Delta(X) definition}. Then we can state our main results.
    
    The first generalizes Fujii's result \eqref{Fujii's result}, since by letting $\alpha\to 0$, \eqref{our first result} reduces to \eqref{Fujii's result}.
    
    \begin{theorem}\label{Theorem: our first result}
    	Fix numbers $X>0$ and $\tau\geqslant|\delta|+1$. Then for $0\ne\delta:=\frac{2\pi\alpha}{\log\frac{T}{2\pi X}}\ll 1$, as $T\to\infty$,
    	\begin{align}
    		\sum_{\tau<\gamma\leqslant T}\zeta'(\rho+i\delta)X^{\rho}={}& -\Delta(X)\frac{T}{2\pi}\left\{X^{-i\delta}\log X\left(\frac{1}{2}\log\frac{T}{2\pi}-\frac{1}{2}+\frac{\pi i}{4}\right)-\sum_{mr=X}\Lambda(r)m^{-i\delta}\log m\right\} \nonumber\\
    		& -X^{1-i\delta}\log X\left(\frac{1}{2}\log X-\frac{\pi i}{4}\right)\sum_{m\leqslant\frac{T}{2\pi X}}e^{2\pi imX} \nonumber\\
    		& +X^{1-i\delta}\log X\left(\sum_{mr\leqslant\frac{T}{2\pi X}}e^{2\pi imrX}\Lambda(r)r^{-i\delta}-\frac{1}{2}\sum_{m\leqslant\frac{T}{2\pi X}}e^{2\pi imX}\log m\right) \nonumber\\
    		& +X^{1-i\delta}\sum_{mr\leqslant\frac{T}{2\pi X}}e^{2\pi imrX}\Lambda(r)r^{-i\delta}\log r+\mo\left(T^{\frac{1}{2}}\log^3 T\right), \label{our first result}
    	\end{align}
    	holds uniformly in $\alpha$.
    \end{theorem}
    
    As Pearce-Crump pointed out in \cite[p. 5]{PeCr}, it is clear from the following corollary that the changing behavior of asymptotic expansions \eqref{our first result} and \eqref{our second result} depends on whether $X$ is an integer.
    
    \begin{corollary}\label{Corollary: our second result}
    	Fix any positive integer $X$ and any number $\tau\geqslant|\delta|+1$. Then for $0\ne\delta:=\frac{2\pi\alpha}{\log\frac{T}{2\pi X}}\ll 1$, as $T\to\infty$,
    	\begin{align}
    		\sum_{\tau<\gamma\leqslant T}\zeta'(\rho+i\delta)X^{\rho}={}& \frac{T}{2\pi}\sum_{mr=X}\Lambda(r)m^{-i\delta}\log m-X^{-i\delta}\,\frac{T}{2\pi}\,Z_{\delta}\left(\frac{T}{2\pi X}\right) \nonumber\\
    		& -X^{-i\delta}\log X\ \frac{T}{2\pi}\left(\log\frac{T}{2\pi}-1+\frac{\zeta'}{\zeta}(1+i\delta)-\frac{\zeta(1-i\delta)}{1-i\delta}\left(\frac{T}{2\pi X}\right)^{-i\delta}\right) \nonumber\\
    		& +E(T) \label{our second result}
    	\end{align}
    	uniformly in $\alpha$, where
    	\[Z_{\delta}(x):=\zeta'(1+i\delta)+\frac{x^{-i\delta}}{1-i\delta}\left\{\frac{\zeta''\zeta-\zeta'^2}{\zeta^2}(1-i\delta)+\frac{\zeta'}{\zeta}(1-i\delta)\Bigl(\log x-\frac{1}{1-i\delta}\Bigr)\right\}\]
    	and the error term is
    	\begin{equation}\label{error term E(T)}
    		E(T)=\begin{cases}
    			\mo\bigl(Te^{-C\sqrt{\log T}}\bigr),& \text{unconditionally}\\
    			\mo\bigl(T^{\frac{1}{2}+\ve}\bigr),& \text{assuming Riemann Hypothesis}
    		\end{cases}
    	\end{equation}
    	with $C>0$ being an absolute constant.
    \end{corollary}
    
    Our second theorem corrects and improves \eqref{JaMa's result}.
    
    \begin{theorem}\label{Theorem: our third result}
    	Fix numbers $a\in\CC,\,X>0$ and $\tau\geqslant |\delta|+1$. Then for $0\ne\delta:=\frac{2\pi\alpha}{\log\frac{T}{2\pi X}}\ll 1$, as $T\to\infty$,
    	\begin{equation}\label{our third result}
    		\begin{split}
    			\sum_{\tau<\gamma_a\leqslant T}\zeta'(\rho_a+i\delta)X^{\rho_a}={}& \sum_{\tau<\gamma\leqslant T}\zeta'(\rho+i\delta)X^{\rho}-\Delta(X)\frac{T}{2\pi}\sum_{mr=X}\Bigl(\Lambda(r)+c_a(r)\Bigr)m^{-i\delta}\log m\\
    			& -aK_{\delta}^{(1)}\left(\frac{T}{2\pi}\right)+\mo\left(T^{\frac{1}{2}}\log^7 T\right)
    		\end{split}
    	\end{equation}
    	uniformly in $\alpha$, where $c_a(1)=0$, the coefficients $c_a(r)\,(r\geqslant 2)$ defined by \cite[Equation (29)]{Steu}
    	\begin{equation}\label{definition of c_a(r)}
    		\frac{\zeta'(s)}{\zeta(s)-a}=\sum_{r\geqslant 2}\frac{c_a(r)}{r^s},\quad \Re(s)>1
    	\end{equation}
    	and
    	\begin{align*}
    		K_{\delta}^{(1)}(x):={}& \Delta\left(X^{-1}\right)\sum_{mn=\frac{1}{X}}\mu(m)n^{i\delta}\cdot x^{1-i\delta}\left(\frac{\log^2 x}{1-i\delta}-\frac{2\log x}{\left(1-i\delta\right)^2}+\frac{2}{\left(1-i\delta\right)^3}\right)\\
    		& +\Delta\left(X^{-1}\right)\sum_{mnr=\frac{1}{X}}\Lambda(m)\left(m^{i\delta}+1\right)n^{i\delta}\mu(r)\cdot x^{1-i\delta}\left(\frac{\log x}{1-i\delta}-\frac{1}{\left(1-i\delta\right)^2}\right)\\
    		& +\Delta\left(X^{-1}\right)\sum_{mnr\ell=\frac{1}{X}}\Lambda(m)\left(mn\right)^{i\delta}\mu(r)\Lambda(\ell)\cdot x^{1-i\delta}\left(1-i\delta\right)^{-1}
    	\end{align*}
    	with $\mu(r)$ denoting the M\"obius function. If, further, $X$ is a positive integer, then \eqref{our third result} reduces to
    	\begin{equation}\label{our third result--X integers}
    		\begin{split}
    			\sum_{\tau<\gamma_a\leqslant T}\zeta'(\rho_a+i\delta)X^{\rho_a}={}& \sum_{\tau<\gamma\leqslant T}\zeta'(\rho+i\delta)X^{\rho}-\frac{T}{2\pi}\sum_{mr=X}\Bigl(\Lambda(r)+c_a(r)\Bigr)m^{-i\delta}\log m\\
    			& -\mathbbm{1}_{X=1}\cdot aK_{\delta}^{(2)}\left(\frac{T}{2\pi}\right)+\mo\left(T^{\frac{1}{2}}\log^7 T\right),
    		\end{split}
    	\end{equation}
    	where
    	\[K_{\delta}^{(2)}(x):=x^{1-i\delta}\left(\frac{\log^2 x}{1-i\delta}-\frac{2\log x}{\left(1-i\delta\right)^2}+\frac{2}{\left(1-i\delta\right)^3}\right).\]
    \end{theorem}
    
    Letting $\delta\to 0$ in \eqref{our third result--X integers} and using \eqref{Fujii's result--0}, one may get the following corollary, which provides the correct statement of \eqref{JaMa's result}.
    
    \begin{corollary}
    	Fix any number $a\in\CC$ and any positive integer $X$. Then
    	\begin{align*}
    		\sum_{1<\gamma_a\leqslant T}\zeta'(\rho_a)X^{\rho_a}={}& \Bigl(\frac{1}{2}-a\cdot\mathbbm{1}_{X=1}\Bigr)\frac{T}{2\pi}\log^2\left(\frac{T}{2\pi}\right)+\bigl(C_0-1-\log X-2a\cdot\mathbbm{1}_{X=1}\bigr)\frac{T}{2\pi}\log\frac{T}{2\pi}\\
    		& +\Biggl\{1-C_0-C_0^2+3C_1-2a\cdot\mathbbm{1}_{X=1}\\
    		& \qquad\ \, -\sum_{mr=X}c_a(r)\log m-\Bigl(C_0-1+\frac{1}{2}\log X\Bigr)\log X\Biggr\}\frac{T}{2\pi}+E(T),
    	\end{align*}
    	where the error term $E(T)$ is given by \eqref{error term E(T)}, the coefficients $c_a(r)$ are defined in \eqref{definition of c_a(r)} and the constants $C_0$ and $C_1$ are the coefficitents in the Laurent expansion of $\zeta(s)$ at $s=1$, namely
    	\begin{equation}\label{Laurent expansion for zeta at 1}
    		\zeta(s)=\frac{1}{s-1}+C_0+C_1(s-1)+\cdots.
    	\end{equation}
    \end{corollary}
	
\section{\bf Overview of the paper}
    In the first two sections we have described the motivation and main results of this paper.
    
    In Section \ref{Section: preliminaries}, we will present some basic results about the Riemann zeta function $\zeta(s)$ and the Riemann xi function $\xi(s)$. We will also make a revision of Garunk\v{s}tis and Steuding's results \cite[Lemmas 4 and 6]{GaSt}.
    
    In Section \ref{Section: proof of first result}, we will use the results from Section \ref{Section: preliminaries} to prove Theorem \ref{Theorem: our first result}. Our starting point is the contour integral
    \[\frac{1}{2\pi i}\int_{\mathcal{C}}\frac{\xi'}{\xi}(s)\zeta'(s+i\delta)X^s\md s,\]
    where $b=1+\frac{1}{\log T},\,\tau\geqslant|\delta|+1$, and $\mathcal{C}$ denotes the postively oriented rectangle with the vertices $b+i\tau,b+iT,1-b+iT,1-b+i\tau$.
    
    In Section \ref{Section: proof of second result}, we will prove Corollary \ref{Corollary: our second result} by evaluating the leading terms.
    
    In Section \ref{Section: proof of third result}, we will carefully evaluate the contour integral
    \[\frac{1}{2\pi i}\int_{\mathcal{R}}\frac{\zeta'(s)}{\zeta(s)-a}\zeta'(s+i\delta)X^s\md s\]
    to prove Theorem \ref{Theorem: our third result}, where $B\asymp\log T,\,b=1+\frac{1}{\log T},\,\tau\geqslant|\delta|+1$, and $\mathcal{R}$ denotes the positively oriented rectangle with the vertices $B+i\tau,B+iT,1-b+iT,1-b+i\tau$.
    
    The paper concludes with some remarks in Section \ref{Section: further remarks} on our main results.
	
\section{\bf Preliminaries}\label{Section: preliminaries}
	In this section, we present some basic results about $\zeta(s)$ and $\xi(s)$ and revisit Garunk\v{s}tis and Steuding's results in \cite{GaSt} about the trivial $a$-points.
	
	We first recall the functional equation for $\zeta(s)$
	\begin{equation}\label{functional equation for zeta}
		\zeta(1-s)=\chi(1-s)\zeta(s),
	\end{equation}
	where
	\[\chi(s)=2^s\pi^{s-1}\sin\Bigl(\frac{\pi s}{2}\Bigr)\Gamma(1-s)\]
	has the leading behavior \cite[Equation (13)]{Gone}
	\begin{equation}\label{chi -- leading behavior}
		\chi(1-s)=e^{-\frac{\pi}{4}i}\left(\frac{t}{2\pi}\right)^{\sigma-\frac{1}{2}}\exp\left(it\log\frac{t}{2\pi e}\right)\left(1+\mo\left(t^{-1}\right)\right)
	\end{equation}
	for $\sigma$ fixed and $t\geqslant 1$, and where $\Gamma(s)$ is the Gamma function.
	
	Taking the logarithmic derivative of \eqref{functional equation for zeta} and using \cite[Equation (33)]{Gone}
	\[\frac{\chi'}{\chi}(1-s)=-\log\frac{|t|}{2\pi}+\mo\bigl(\left|t\right|^{-1}\bigr)\]
	imply that for $\sigma$ fixed and $|t|\geqslant t_0>0$,
	\begin{equation}\label{functional equation for zeta'/zeta}
		\frac{\zeta'}{\zeta}(1-s)=-\log\frac{|t|}{2\pi}-\frac{\zeta'}{\zeta}(s)+\mo\bigl(\left|t\right|^{-1}\bigr).
	\end{equation}
	
	It should be mentioned that for $\sigma>1$, the following Dirichlet series are convergent:
	\begin{equation}\label{Dirichlet series for zeta-functions}
		\frac{1}{\zeta(s)}=\sum_{k=1}^{\infty}\frac{\mu(k)}{k^s},\ \ \zeta^{(n)}(s)=\left(-1\right)^n\sum_{m=1}^{\infty}\frac{\log^n m}{m^s},\ \ \frac{\zeta'}{\zeta}(s)=-\sum_{r=1}^{\infty}\frac{\Lambda(r)}{r^s}.
	\end{equation}
	We also need the following lemmas.
	\begin{lemma}[{\cite[Lemma 1]{HuPC}}]\label{Lemma: convexity bound for zeta}
		For $t\geqslant t_0>0$ uniformly in $\sigma$,
		\[\zeta^{(n)}(\sigma+it)\ll\begin{cases}
			t^{\frac{1}{2}-\sigma}\log^{n+1}t,& \text{if }\sigma\leqslant 0,\\
			t^{\frac{1}{2}(1-\sigma)}\log^{n+1}t,& \text{if }0\leqslant\sigma\leqslant 1,\\
			\log^{n+1}t,& \text{if }\sigma\geqslant 1.
		\end{cases}\]
	\end{lemma}
	
	\begin{lemma}[{\cite[Lemma 2]{HuPC}}]\label{Lemma: upper bound of zeta'/zeta}
		If $T$ is such that $|T-\gamma|\gg\frac{1}{\log T}$ for any ordinate $\gamma$, uniformly for $-1\leqslant \sigma\leqslant 2$ we have
		\[\left(\frac{\zeta'}{\zeta}\right)^{(n)}(\sigma+iT)\ll\log^{n+2}T.\]
	\end{lemma}
	
	\begin{lemma}[{\cite[Lemma 4]{HuPC}}]\label{Lemma: functional equation for derivatives of zeta}
		For $\sigma\geqslant 1$ and $t\geqslant 1$,
		\[\zeta^{(n)}(1-s)=\left(-1\right)^n\chi(1-s)\sum_{k=0}^{n}\binom{n}{k}\log^{n-k}\Bigl(\frac{t}{2\pi}\Bigr)\zeta^{(k)}(s)+\mo\left(t^{\sigma-\frac{3}{2}}\log^{n-1}t\right).\]
	\end{lemma}
	
	\begin{lemma}[{\cite[Lemma 4]{PeCr}}]\label{Lemma: PeCr's lemma}
		Let $X>0$ be a fixed number. Let $\{b_m\}_{m=1}^{\infty}$ be a sequence of complex numbers such that for any $\ve>0$, $b_m\ll m^{\ve}$. Let $c>1$ and let $k\geqslant 0$ be an integer. Then for $T$ suffciently large, we have
		\begin{align*}
			\frac{1}{2\pi}& \int_1^T\chi(1-c-it)\left(\sum_{m=1}^{\infty}b_m m^{-c-it}\right)\log^k\left(\frac{t}{2\pi}\right)e^{-it\log X}\md t\\
			& =X^c\sum_{1\leqslant m\leqslant \frac{T}{2\pi}}b_m\log^k\left(mX\right)e^{-2\pi imX}+\mo\left(T^{c-\frac{1}{2}}\log^k T\right).
		\end{align*}
	\end{lemma}
	
	Recall the Riemann xi function defined by
	\[\xi(s)=\frac{1}{2}s(s-1)\pi^{-\frac{s}{2}}\Gamma\left(\frac{s}{2}\right)\zeta(s),\]
	which satisfies the functional equation $\xi(s)=\xi(1-s)$. We need the following expansion for $\frac{\xi'}{\xi}(s)$.
	
	\begin{lemma}[{\cite[Lemma 3]{PeCr}}]\label{Lemma: expansion of xi'/xi}
		For $0<\arg(s)<\pi-\theta,\,0<\theta<\pi$ and $t\geqslant t_0>0$,
		\[\frac{\xi'}{\xi}(s)=\frac{1}{2}\log\frac{t}{2\pi}+\frac{\pi i}{4}+\frac{\zeta'}{\zeta}(s)+\mo\left(t^{-1}\right).\]
	\end{lemma}
	
	Now let us reformulate Garunk\v{s}tis and Steuding's results on the location of the trivial $a$-points.
	
	\begin{lemma}[{\cite[Lemma 4]{GaSt}}]\label{Lemma: lower bound of zeta}
		There is a constant $c>0$ such that
		\[|\zeta(s)|>c\left(2\pi e\right)^{\sigma}\frac{\left|1-s\right|^{\frac{1}{2}-\sigma}}{\log|t|}\qquad (\sigma\leqslant 0,|t|\geqslant 2).\]
		If Riemann Hypothesis is true, fix $\sigma<\frac{1}{2}$ and then for any $\epsilon>0$, there is $c=c(\epsilon)>0$ such that
		\[|\zeta(s)|>c\left(2\pi e\right)^{\sigma}\left|t\right|^{\frac{1}{2}-\sigma+\epsilon}\qquad \bigl(\sigma<1/2,|t|\geqslant 2\bigr).\]
	\end{lemma}
	
	\begin{proof}
		It is well known (see \cite[p. 65, Exercise 4.6]{Patt}) that there is a constant $c>0$ such that, for $1-\frac{c}{\log(|t|+2)}<\sigma<2$,
		\[\zeta\left(s\right)^{-1}=\mo(\log|s|),\quad |s|\to\infty.\]
		Thus, for $1\leqslant \sigma\leqslant 1+\frac{1}{\log|t|}$ and $|t|\geqslant 2$,
		\[|\zeta(\sigma+it)|\gg \frac{2}{\log\left(\sigma^2+t^2\right)}\gg \frac{1}{\log|t|}.\]
		Recall the estimate \cite[p. 50, Equation (3.6.1)]{Titc}
		\[\zeta\left(\sigma+it\right)^{-1}\ll \frac{\log^{\frac{1}{4}}|t|}{\left(\sigma-1\right)^{\frac{3}{4}}}\qquad (\sigma>1,|t|\geqslant 2),\]
		which shows that $|\zeta(\sigma+it)|\gg \frac{1}{\log|t|}$ holds for $\sigma>1+\frac{1}{\log|t|}$ and $|t|\geqslant 2$. Hence
		\begin{equation}\label{lower bound of zeta--1}
			|\zeta(\sigma+it)|\gg \frac{1}{\log|t|}\qquad (\sigma\geqslant 1,|t|\geqslant 2).
		\end{equation}
		
		Assuming Riemann Hypothesis, we get from \cite[p. 337, Equation (14.2.6)]{Titc} that for fixed $\sigma>\frac{1}{2}$,
		\begin{equation}\label{lower bound of zeta--2}
			|\zeta(\sigma+it)|\gg_{\sigma,\epsilon}\left|t\right|^{-\epsilon}\qquad (|t|\geqslant 2).
		\end{equation}
		Use of Stirling's formula \cite[Equation (5.11.1)]{NIST-Handbook} confirms that for $\sigma<\frac{1}{2}$,
		\begin{equation}\label{lower bound of chi}
			|\chi(s)|\asymp \left(2\pi e\right)^{\sigma}\left|1-s\right|^{\frac{1}{2}-\sigma}e^{|t|\arctan\frac{1-\sigma}{|t|}}\geqslant \left(2\pi e\right)^{\sigma}\left|1-s\right|^{\frac{1}{2}-\sigma}.
		\end{equation}
		Combining this with \eqref{lower bound of zeta--1} and \eqref{lower bound of zeta--2} gives the derised results.
	\end{proof}
	
	\begin{lemma}[{\cite[Lemma 6]{GaSt}}]\label{Lemma: location of trivial a-points}
		For any complex number $a$, there is a positive integer $N$ such that there is just one simple $a$-point of $\zeta(s)$ in a neighborhood around $s=-2k$ for all positive integers $k\geqslant N$; apart from these, there are no other $a$-points in the left half-plane $\Re(s)\leqslant 0$ except possibly finitely many near $s=0$.
	\end{lemma}
	
	\begin{proof}
		Fix $a\in\CC$. Lemma \ref{Lemma: lower bound of zeta} claims that there exists a constant $T_a>0$ such that the inequality $|\zeta(s)|\geqslant|a|+1$ holds throughout the domain $\{s\in\CC:\sigma\leqslant 0,|t|\geqslant T_a\}$.
		
		For any positive integer $n$, consider the region
		\[\Omega_k:=\{s\in\CC:-2k-1<\sigma<-2k+1,\,|t|<T_a\}.\]
		By \eqref{lower bound of chi}, we obtain that there is a positive integer $N\geqslant 2\pi e$ such that for $k\geqslant N$, $|\zeta(s)|\geqslant |a|+1$ is valid on the boundary $\partial\Omega_k$. Applying Rouch\'e's theorem, $\zeta(s)-a$ and $\zeta(s)$ have the same number of zeros in $\Omega_k$ for each $k\geqslant N$. Moreover, $\zeta(s)-a$ has finitely many zeros in $\{s\in\CC:-2N+1<\sigma\leqslant 0,|t|<T_a\}$ since it is analytic there.
	\end{proof}
	
	\begin{remark}
		For any fixed $c>0$, apply {\rm\eqref{lower bound of chi}} and Rouch\'e's theorem in the region
		\[\Omega_k':=\{s\in\CC:-2k-1<\sigma<-2k+1,\,|t|<c\}.\]
		Then there exists an $N>1$ such that $\zeta(s)-a$ and $\zeta(s)$ have the same number of zeros in $\Omega_k'$ for each $k\geqslant N$. Thus there are finitely many $a$-points in $\{s\in\CC:\sigma\leqslant 0\}\sm \bigcup_{k\geqslant N}\Omega_k'$. In particular,
		\begin{equation}\label{number of trivial a-points for t>1}
			\sum_{\substack{\beta_a\leqslant 0\\c<\gamma_a\leqslant T}}1=\mo(1).
		\end{equation}
	\end{remark}
	
	The following partial fraction decomposition formula is a direct consequence of Lemma \ref{Lemma: location of trivial a-points}; see \cite[p. 8]{GaSt} for the proof. 
	
	\begin{lemma}[{\cite[Lemma 8]{GaSt}}]\label{Lemma: partial fraction decomposition}
		Let $a$ be a fixed complex number. Then, for $-1\leqslant \sigma\leqslant 2$ and $|t|\geqslant 1$,
		\[\frac{\zeta'(s)}{\zeta(s)-a}=\sum_{|t-\gamma_a|\leqslant 1}\frac{1}{s-\rho_a}+\mo(\log(|t|+1)),\]
		where the summation is taken over all $a$-points $\rho_a=\beta_a+i\gamma_a$ satisfying $|t-\gamma_a|\leqslant 1$.
	\end{lemma}
	
\section{\bf Proof of Theorem \ref{Theorem: our first result}}\label{Section: proof of first result}
    Fix numbers $X>0$ and $\tau\geqslant|\delta|+1$. Write $s=\sigma+it$ with $\sigma,t\in\RR$ and $\rho=\beta+i\gamma$ for a non-trivial zero of $\zeta(s)$. Suppose that $T$ is large. Write
    \[I=\sum_{\tau<\gamma\leqslant T}\zeta'(\rho+i\delta)X^{\rho}.\]
    
    Set $b=1+\frac{1}{\log T}$. Then use of Cauchy's residue theorem yields
    \begin{align}
    	I& =\frac{1}{2\pi i}\left(\int_{1-b+i\tau}^{b+i\tau}+\int_{b+i\tau}^{b+iT}+\int_{b+iT}^{1-b+iT}+\int_{1-b+iT}^{1-b+i\tau}\right)\frac{\xi'}{\xi}(s)\zeta'(s+i\delta)X^s\md s+\mo\left(\log^3 T\right) \label{initial error term for I}\\
    	& =:I_B+I_R+I_T+I_L+\mo\left(\log^3 T\right).\nonumber
    \end{align}
    Throughout this section, we assume that $\left|T-\gamma\right|\gg\frac{1}{\log T}$. In view of \eqref{Riemann-von Mangoldt formula} with $a=0$, the error $\mo(\log^3 T)$ in \eqref{initial error term for I} results from removing this restriction on $T$.
    
    Using the upper bound for $\frac{\xi'}{\xi}(s)$ \cite[Equation (18)]{Gone}
    \[\frac{\xi'}{\xi}(\sigma+iT)\ll \log^2 T\qquad (-1\leqslant\sigma\leqslant 2)\]
    and the estimate \cite[Lemma 2]{PeCr}
    \[\int_{1-b}^b\left|\zeta'(\sigma+iT)\right|\md\sigma\ll T^{\frac{1}{2}}\log T,\]
    we get
    \begin{equation}\label{estimate for I_B+I_T}
    	I_B+I_T\ll 1+\log^2 T\int_{1-b}^b\left|\zeta'(\sigma+iT)\right|X^{\sigma}\md\sigma\ll T^{\frac{1}{2}}\log^3 T.
    \end{equation}

\subsection{Evaluating $I_R$}
    \ 

    Using Lemma \ref{Lemma: expansion of xi'/xi}, we have
    \begin{align*}
    	I_R& =\frac{1}{2\pi}\int_{\tau}^T\frac{\xi'}{\xi}(b+it)\zeta'(b+it+i\delta)X^{b+it}\md t\\
    	& =\frac{X^b}{2\pi}\int_{\tau}^T\left\{\left(\frac{1}{2}\log\frac{t}{2\pi}+\frac{\pi i}{4}+\mo\left(t^{-1}\right)\right)+\frac{\zeta'}{\zeta}(b+it)\right\}\zeta'(b+it+i\delta)X^{it}\md t\\
    	& =:I_{R,1}+I_{R,2}.
    \end{align*}
    Using \eqref{Dirichlet series for zeta-functions} and Lemma \ref{Lemma: convexity bound for zeta} gives
    \[I_{R,1}=-\frac{X^b}{2\pi}\sum_{m\geqslant 1}\frac{\log m}{m^{b+i\delta}}\int_{\tau}^T\left(\frac{1}{2}\log\frac{t}{2\pi}+\frac{\pi i}{4}\right)e^{it\log\frac{X}{m}}\md t+\mo\left(\log^3 T\right).\]
    For terms with $\frac{X}{m}\ne 1$ the integral can be estimated by integration by parts. Due to \eqref{Laurent expansion for zeta at 1}, these terms contribute an error term
    \[\ll\log T\sum_{m\geqslant 1}\frac{\log m}{m^b}=\left|\frac{\zeta'}{\zeta}(b)\right|\log T\ll\log^2 T.\]
    Thus, with $\Delta(X)$ defined as in \eqref{Delta(X) definition},
    \begin{align*}
    	I_{R,1}& =-X^b\Delta(X)X^{-b-i\delta}\log X \int_{\tau}^T\left(\frac{1}{2}\log\frac{t}{2\pi}+\frac{\pi i}{4}\right)\md t+\mo\left(\log^3 T\right)\\
    	& =-\Delta(X)X^{-i\delta}\log X\ \frac{T}{2\pi}\left(\frac{1}{2}\log\frac{T}{2\pi}-\frac{1}{2}+\frac{\pi i}{4}\right)+\mo\left(\log^3 T\right).
    \end{align*}
    Similarly,
    \[I_{R,2}=\Delta(X)\frac{T}{2\pi}\sum_{mn=X}\Lambda(m)n^{-i\delta}\log n+\mo\left(\log^3 T\right).\]
    Hence
    \[I_R=-\Delta(X)\frac{T}{2\pi}\left\{X^{-i\delta}\log X\left(\frac{1}{2}\log\frac{T}{2\pi}-\frac{1}{2}+\frac{\pi i}{4}\right)-\sum_{mr=X}\Lambda(r)m^{-i\delta}\log m\right\}+\mo\left(\log^3 T\right).\]
    
\subsection{Evaluating $I_L$}
    \ 
    
    Note that
    \begin{align*}
    	I_L& =-\frac{1}{2\pi i}\int_{1-b+i\tau}^{1-b+iT}\frac{\xi'}{\xi}(s)\zeta'(s+i\delta)X^s\md s\\
    	& =\frac{1}{2\pi i}\int_{1-b+i\tau}^{1-b+iT}\frac{\xi'}{\xi}(1-s)\zeta'(s+i\delta)X^s\md s\\
    	& =\frac{1}{2\pi i}\int_{b-i\tau}^{b-iT}\frac{\xi'}{\xi}(u)\zeta'(1-u+i\delta)X^{1-u}\md u.
    \end{align*}
    Applying Lemmas \ref{Lemma: functional equation for derivatives of zeta} and \ref{Lemma: expansion of xi'/xi}, we get
    \begin{align*}
    	-\overline{I_L}={}& \frac{X^{1-b+i\delta}}{2\pi}\int_{\tau+\delta}^{T+\delta}\frac{\xi'}{\xi}(b+it-i\delta)\\
    	& \qquad\qquad\qquad\ \ \biggl\{\chi(1-b-it)\biggl(\zeta(b+it)\log\frac{t}{2\pi} +\zeta'(b+it)\biggr)+\mo\left(t^{\sigma-\frac{3}{2}}\log t\right)\biggr\}X^{-it}\md t\\
    	={}& \frac{X^{1-b+i\delta}}{2\pi}\int_{\tau+\delta}^{T+\delta}\biggl(\frac{1}{2}\log\frac{t}{2\pi}+\frac{\pi i}{4}+\frac{\zeta'}{\zeta}(b+it-i\delta)+\mo\left(t^{-1}\right)\biggr)\\
    	& \qquad\qquad\qquad\ \ \biggl\{\chi(1-b-it)\biggl(\zeta(b+it)\log\frac{t}{2\pi} +\zeta'(b+it)\biggr)+\mo\left(t^{\sigma-\frac{3}{2}}\log t\right)\biggr\}e^{-it\log X}\md t\\
    	={}& \frac{X^{1-b+i\delta}}{2\pi}\int_{\tau+\delta}^{T+\delta}\chi(1-b-it)\biggl(\frac{1}{2}\log\frac{t}{2\pi}+\frac{\pi i}{4}+\frac{\zeta'}{\zeta}(b+it-i\delta)\biggr)\\
    	& \qquad\qquad\qquad\ \ \biggl(\zeta(b+it)\log\frac{t}{2\pi} +\zeta'(b+it)\biggr)e^{-it\log X}\md t+\mo\left(T^{\frac{1}{2}}\log^3 T\right)\\
    	=:{}& J_1+J_2+J_3+J_4+\mo\left(T^{\frac{1}{2}}\log^3 T\right).
    \end{align*}
    Applying Lemma \ref{Lemma: PeCr's lemma} to each of the integrals $J_k$, we have
    \begin{align*}
    	J_1& =\frac{X^{1-b+i\delta}}{2\pi}\int_{\tau+\delta}^{T+\delta}\chi(1-b-it)\zeta(b+it)\log\frac{t}{2\pi}\biggl(\frac{1}{2}\log\frac{t}{2\pi}+\frac{\pi i}{4}\biggr)e^{-it\log X}\md t\\
    	& =\frac{X^{1-b+i\delta}}{2\pi}\int_{\tau+\delta}^{T+\delta}\chi(1-b-it)\sum_{m\geqslant 1}m^{-b-it}\cdot\log\frac{t}{2\pi}\biggl(\frac{1}{2}\log\frac{t}{2\pi}+\frac{\pi i}{4}\biggr)e^{-it\log X}\md t\\
    	& =X^{1+i\delta}\sum_{m\leqslant\frac{T+\delta}{2\pi X}}e^{-2\pi imX}\biggl(\frac{1}{2}\log^2(mX)+\frac{\pi i}{4}\log(mX)\biggr)+\mo\left(T^{\frac{1}{2}}\log^2 T\right)\\
    	& =X^{1+i\delta}\sum_{m\leqslant\frac{T}{2\pi X}}e^{-2\pi imX}\biggl(\frac{1}{2}\log^2(mX)+\frac{\pi i}{4}\log(mX)\biggr)+\mo\left(T^{\frac{1}{2}}\log^2 T\right),\\
    	J_2& =\frac{X^{1+i\delta}}{2\pi}\int_{\tau+\delta}^{T+\delta}\chi(1-b-it)\zeta(b+it)\frac{\zeta'}{\zeta}(b+it-i\delta)e^{-it\log X}\log\frac{t}{2\pi}\md t\\
    	& =-\frac{X^{1-b+i\delta}}{2\pi}\int_{\tau+\delta}^{T+\delta}\chi(1-b-it)\sum_{m,n\geqslant 1}\frac{\Lambda(n)n^{i\delta}}{\left(mn\right)^{b+it}}\cdot e^{-it\log X}\log\frac{t}{2\pi}\md t\\
    	& =-X^{1+i\delta}\sum_{mn\leqslant\frac{T+\delta}{2\pi X}}e^{-2\pi imnX}\Lambda(n)n^{i\delta}\log(mnX)+\mo\left(T^{\frac{1}{2}}\log T\right)\\
    	& =-X^{1+i\delta}\sum_{mn\leqslant\frac{T}{2\pi X}}e^{-2\pi imnX}\Lambda(n)n^{i\delta}\log(mnX)+\mo\left(T^{\frac{1}{2}}\log T\right),\\
    	J_3& =\frac{X^{1-b+i\delta}}{2\pi}\int_{\tau+\delta}^{T+\delta}\chi(1-b-it)\zeta'(b+it)\biggl(\frac{1}{2}\log\frac{t}{2\pi}+\frac{\pi i}{4}\biggr)e^{-it\log X}\md t\\
    	& =-\frac{X^{1-b+i\delta}}{2\pi}\int_{\tau+\delta}^{T+\delta}\chi(1-b-it)\sum_{m\geqslant 1}\frac{\log m}{m^{b+it}}\cdot\biggl(\frac{1}{2}\log\frac{t}{2\pi}+\frac{\pi i}{4}\biggr)e^{-it\log X}\md t\\
    	& =-X^{1+i\delta}\sum_{m\leqslant\frac{T+\delta}{2\pi X}}e^{-2\pi imX}\biggl(\frac{1}{2}\log(mX)+\frac{\pi i}{4}\biggr)\log m+\mo\left(T^{\frac{1}{2}}\log T\right)\\
    	& =-X^{1+i\delta}\sum_{m\leqslant\frac{T}{2\pi X}}e^{-2\pi imX}\biggl(\frac{1}{2}\log(mX)+\frac{\pi i}{4}\biggr)\log m+\mo\left(T^{\frac{1}{2}}\log T\right),\\
    	J_4& =\frac{X^{1+i\delta}}{2\pi}\int_{\tau+\delta}^{T+\delta}\chi(1-b-it)\zeta'(b+it)\frac{\zeta'}{\zeta}(b+it-i\delta)e^{-it\log X}\md t\\
    	& =\frac{X^{1-b+i\delta}}{2\pi}\int_{\tau+\delta}^{T+\delta}\chi(1-b-it)\sum_{m,n\geqslant 1}\frac{\Lambda(n)n^{i\delta}\log m}{\left(mn\right)^{b+it}}\cdot e^{-it\log X}\md t\\
    	& =X^{1+i\delta}\sum_{mn\leqslant\frac{T+\delta}{2\pi X}}e^{-2\pi imnX}\Lambda(n)n^{i\delta}\log m+\mo\left(T^{\frac{1}{2}}\right)\\
    	& =X^{1+i\delta}\sum_{mn\leqslant\frac{T}{2\pi X}}e^{-2\pi imnX}\Lambda(n)n^{i\delta}\log m+\mo\left(T^{\frac{1}{2}}\right).
    \end{align*}
    
    Combining these expressions, we have
    \begin{align*}
    	\overline{I_L}={}& -X^{1+i\delta}\log X\biggl(\frac{1}{2}\log X+\frac{\pi i}{4}\biggr)\sum_{m\leqslant\frac{T}{2\pi X}}e^{-2\pi imX}-\frac{1}{2}X^{1+i\delta}\log X\sum_{m\leqslant\frac{T}{2\pi X}}e^{-2\pi imX}\log m\\
    	& +X^{1+i\delta}\sum_{mn\leqslant\frac{T}{2\pi X}}e^{-2\pi imnX}\Lambda(n)n^{i\delta}\log(nX)+\mo\left(T^{\frac{1}{2}}\log^3 T\right).
    \end{align*}
    Hence
    \begin{align*}
    	I_L={}& -X^{1-i\delta}\log X\biggl(\frac{1}{2}\log X-\frac{\pi i}{4}\biggr)\sum_{m\leqslant\frac{T}{2\pi X}}e^{2\pi imX}-\frac{1}{2}X^{1-i\delta}\log X\sum_{m\leqslant\frac{T}{2\pi X}}e^{2\pi imX}\log m\\
    	& +X^{1-i\delta}\sum_{mn\leqslant\frac{T}{2\pi X}}e^{2\pi imnX}\Lambda(n)n^{-i\delta}\log(nX)+\mo\left(T^{\frac{1}{2}}\log^3 T\right).
    \end{align*}
    
\subsection{Finishing the proof}
    \ 
    
    Combining \eqref{initial error term for I}, \eqref{estimate for I_B+I_T} and the expansions for $I_R$ and $I_L$ gives
    \begin{align*}
    	I={}& -\Delta(X)\frac{T}{2\pi}\left\{X^{-i\delta}\log X\left(\frac{1}{2}\log\frac{T}{2\pi}-\frac{1}{2}+\frac{\pi i}{4}\right)-\sum_{mr=X}\Lambda(r)m^{-i\delta}\log m\right\}\\
    	& -X^{1-i\delta}\log X\biggl(\frac{1}{2}\log X-\frac{\pi i}{4}\biggr)\sum_{m\leqslant\frac{T}{2\pi X}}e^{2\pi imX}-\frac{1}{2}X^{1-i\delta}\log X\sum_{m\leqslant\frac{T}{2\pi X}}e^{2\pi imX}\log m\\
    	& +X^{1-i\delta}\sum_{mr\leqslant\frac{T}{2\pi X}}e^{2\pi imnX}\Lambda(r)r^{-i\delta}\log(rX)+\mo\left(T^{\frac{1}{2}}\log^3 T\right).
    \end{align*}
    This completes the proof of Theorem \ref{Theorem: our first result}.

\section{\bf Proof of Corollary \ref{Corollary: our second result}}\label{Section: proof of second result}
    Assume that $X$ is a positive integer. Then the right-hand side of \eqref{our first result} reads
    \begin{equation}\label{leading terms for second result}
    	\begin{split}
    		& -\frac{T}{2\pi}\left\{X^{-i\delta}\log X\left(\frac{1}{2}\log\frac{T}{2\pi}-\frac{1}{2}+\frac{\pi i}{4}\right)-\sum_{mr=X}\Lambda(r)m^{-i\delta}\log m\right\}\\
    		& -X^{1-i\delta}\log X\left(\frac{1}{2}\log X-\frac{\pi i}{4}\right)\sum_{m\leqslant\frac{T}{2\pi X}}1-\frac{1}{2}X^{1-i\delta}\log X\sum_{m\leqslant\frac{T}{2\pi X}}\log m\\
    		& +X^{1-i\delta}\log X\sum_{mr\leqslant\frac{T}{2\pi X}}\Lambda(r)r^{-i\delta}+X^{1-i\delta}\sum_{mr\leqslant\frac{T}{2\pi X}}\Lambda(r)r^{-i\delta}\log r+\mo\left(T^{\frac{1}{2}}\log^3 T\right).
    	\end{split}
    \end{equation}
    
    Applying the truncated Perron formula \cite[Corollary 5.3]{MoVa} and following the analysis in \cite[Sections 6 and 7]{HuPC}, we may obtain
    \begin{align*}
    	L:=\sum_{mr\leqslant x}\Lambda(r)r^{i\delta}\log r=\sum_{s_0\in\{1,1+i\delta\}}\underset{s=s_0}{\text{Res}}\left(\frac{\zeta'}{\zeta}(s)\zeta'(s-i\delta)\frac{x^s}{s}\right)+E(x),
    \end{align*}
    where $E(x)$ is defined as in \eqref{error term E(T)}. Calculation of the residues on the right gives
    \begin{equation}\label{calculation of residues}
    	L=-\zeta'(1-i\delta)x-\frac{x^{1+i\delta}}{1+i\delta}\left\{\frac{\zeta''\zeta-\zeta'^2}{\zeta^2}(1+i\delta)+\frac{\zeta'}{\zeta}(1+i\delta)\left(\log x-\frac{1}{1+i\delta}\right)\right\}+E(x).
    \end{equation}
    
    Note that
    \begin{equation}
    	\sum_{m\leqslant\frac{T}{2\pi X}}\log m=\frac{T}{2\pi X}\log\frac{T}{2\pi X}-\frac{T}{2\pi X}+\mo\left(\log T\right)
    \end{equation}
    and recall Fujii's estimate \cite[p. 112]{Fuji--1}
    \begin{equation}\label{Fujii's result--1}
    	\sum_{mr\leqslant x}\Lambda(r)r^{i\delta}=\frac{\zeta(1+i\delta)}{1+i\delta}x^{1+i\delta}-\frac{\zeta'}{\zeta}(1-i\delta)x+E(x).
    \end{equation}
    Combining the formulas \eqref{leading terms for second result}-\eqref{Fujii's result--1} completes the proof.

\section{\bf Proof of Theorem \ref{Theorem: our third result}}\label{Section: proof of third result}
    Fix the numbers $a\in\CC,\,X>0$ and $\tau\geqslant|\delta|+1$, let $0\ne \delta=\frac{2\pi\alpha}{\log\frac{T}{2\pi X}}\ll 1$ and assume that $T$ is large. Write $s=\sigma+it$ with $\sigma,t\in\RR$ and $\rho_a=\beta_a+i\gamma_a$ for an $a$-point of $\zeta(s)$. Set
    \[S:=\sum_{\tau<\gamma_a\leqslant T}\zeta'(\rho_a+i\delta)X^{\rho_a}.\]
    If we assume
    \begin{equation}\label{restriction on T for rho_a}
    	\min_{\rho_a}\left|T-\gamma_a\right|\geqslant\frac{1}{\log T},
    \end{equation}
    use of \eqref{Riemann-von Mangoldt formula} and Lemmas \ref{Lemma: convexity bound for zeta} and \ref{Lemma: location of trivial a-points} yields
    \[S=\sum_{\substack{\tau<\gamma_a\leqslant T\\\text{assuming \eqref{restriction on T for rho_a}}}}\zeta'(\rho_a+i\delta)X^{\rho_a}+\mo\left(T^{\frac{1}{2}}\log^3 T\right).\]
    
    Assume that $a\ne 1$. Now we choose $b=1+\frac{1}{\log T}$ and $B=\frac{\log T}{2\log(X+2)}$. Since
    \[\zeta(\sigma+it)=1+o(1),\quad\sigma\to +\infty\]
    uniformly in $t$, there are no $a$-points in the half-plane $\Re(s)>B-1$. Take $\mathcal{R}$ to be the positively oriented rectangle with the vertices $1-b+i\tau,B+i\tau,B+iT,1-b+iT$. If assuming \eqref{restriction on T for rho_a}, in view of \eqref{number of trivial a-points for t>1} and Lemma \ref{Lemma: location of trivial a-points}, we may take $\tau$ large enough such that there are no $a$-points on the boundary $\partial\mathcal{R}$. Thus, under \eqref{restriction on T for rho_a},
    \[\sum_{\substack{\tau<\gamma_a\leqslant T\\\text{assuming \eqref{restriction on T for rho_a}}}}\zeta'(\rho_a+i\delta)X^{\rho_a}=\frac{1}{2\pi i}\int_{\mathcal{R}}\frac{\zeta'(s)}{\zeta(s)-a}\zeta'(s+i\delta)X^s\md s.\]
    
    If $a=1$, consider the function $f(s)=2^s(\zeta(s)-1)$ in place of $\zeta(s)-a$. Since
    \[f(s)=1+\sum_{n\geqslant 3}\left(\frac{2}{n}\right)^s,\]
    there is a zero-free right half-plane for $f(s)$. Furthermore,
    \[\frac{f'}{f}(s)=\log 2+\frac{\zeta'(s)}{\zeta(s)-1}.\]
    It implies that the constant term does not contribute by integration over a closed contour and we can use the same argument as in the case $a\ne 1$.
    
    In the rest of this section, we assume \eqref{restriction on T for rho_a} and let $a\in\CC\sm\{1\}$.
    
\subsection{Beginning the proof}
    \ 

    We have
    \begin{align*}
    	S& =\frac{1}{2\pi i}\int_{\mathcal{R}}\frac{\zeta'(s)}{\zeta(s)-a}\zeta'(s+i\delta)X^s\md s+\mo\left(T^{\frac{1}{2}}\log^3 T\right)\\
    	& =\frac{1}{2\pi i}\left(\int_{1-b+iT}^{1-b+i\tau}+\int_{1-b+i\tau}^{B+i\tau}+\int_{B+i\tau}^{B+iT}+\int_{B+iT}^{1-b+iT}\right)\frac{\zeta'(s)}{\zeta(s)-a}\zeta'(s+i\delta)X^s\md s+\mo\left(T^{\frac{1}{2}}\log^3 T\right)\\
    	& =:S^L+S^B+S^R+S^T+\mo\left(T^{\frac{1}{2}}\log^3 T\right).
    \end{align*}
    
    The integral along the bottom of $\mathcal{R}$ is
    \[S^B=\frac{1}{2\pi i}\left(\int_{1-b+i\tau}^{2+i\tau}+\int_{2+i\tau}^{B+i\tau}\right)\frac{\zeta'(s)}{\zeta(s)-a}\zeta'(s+i\delta)X^s\md s\ll 1+\int_2^B X^{\sigma}\md \sigma.\]
    We obtain from the following lemma that $S^B=\mo(T^{\frac{1}{2}})$.
    
    \begin{lemma}\label{Lemma: integral on [2,B]}
    	If $X>0$ and $B=\frac{\log T}{2\log(X+2)}$, then $\int_2^B X^{\sigma}\md\sigma\ll T^{\frac{1}{2}}$.
    \end{lemma}
    \begin{proof}
    	If X=1, then $\int_2^B X^{\sigma}\md\sigma=B-2\ll \log T$. If $X\ne 1$, $\frac{\log X}{2\log(X+2)}<\frac{1}{2}$ and thus
    	\[\int_2^B X^{\sigma}\md\sigma=\frac{X^B-X^2}{\log X}\ll X^B=\exp\left(\frac{\log X}{2\log(X+2)}\log T\right)<T^{\frac{1}{2}},\]
    	which completes the proof.
    \end{proof}
    
    To estimate $S^R$, note that as $\sigma\to+\infty$,
    \[\zeta'(s)=-\sum_{m=2}^{N}\frac{\log m}{m^s}+\mo\left(\left(N+1\right)^{-\sigma}\right).\]
    Applying \eqref{definition of c_a(r)}, we obtain that as $\sigma\to+\infty$,
    \begin{align*}
    	\frac{\zeta'(s)}{\zeta(s)-a}\zeta'(s+i\delta)&=\left(\sum_{2\leqslant m\leqslant N}\frac{-\log m}{m^{s+i\delta}}+\mo\left(\left(N+1\right)^{-\sigma}\right)\right)\left(\sum_{2\leqslant r\leqslant N}\frac{c_a(r)}{r^s}+\mo\left(\left(N+1\right)^{-\sigma}\right)\right)\\
    	& =\sum_{4\leqslant k\leqslant N}\frac{a_k}{k^s}+\mo\left(\left(N+1\right)^{-\sigma}\right),
    \end{align*}
    where
    \[a_k=-\sum_{mr=k}c_a(r)m^{-i\delta}\log m.\]
    Take $N=\left(\lfloor X\rfloor+3\right)^5-1$ and then for $s$ on the segment $[B+i\tau,B+iT]$,
    \[\frac{\zeta'(s)}{\zeta(s)-a}\zeta'(s+i\delta)X^s=\sum_{4\leqslant k\leqslant N}a_k k^{-s}X^s+\mo\left(X^B\left(N+1\right)^{-B}\right),\]
    which illustrates that
    \begin{align*}
    	S^R& =\frac{1}{2\pi i}\sum_{4\leqslant k\leqslant N}a_k\int_{B+i\tau}^{B+iT}X^s k^{-s}\md s+\mo\left(TX^B\left(N+1\right)^{-B}\right)\\
    	& =\Delta(X)\frac{a_X T}{2\pi}+\mo\left(4^{-B}X^B\right)+\mo\left(TX^B\left(N+1\right)^{-B}\right).
    \end{align*}
    Since
    \begin{align*}
    	4^{-B}X^B& =\exp\left(\frac{\log X-\log 4}{2\log(X+2)}\log T\right)<T^{\frac{1}{2}},\\
    	TX^B\left(N+1\right)^{-B}& <T\left(\lfloor X\rfloor+3\right)^{-4B}=\exp\left(\log T-\frac{2\log(\lfloor X\rfloor+3)}{\log(X+2)}\log T\right)<T^{-1},
    \end{align*}
    we have
    \[S^R=-\Delta(X)\frac{T}{2\pi}\sum_{mr=k}c_a(r)m^{-i\delta}\log m+\mo\left(T^{\frac{1}{2}}\right).\]
    
    Now estimate $S^T$. It follows from Lemma \ref{Lemma: partial fraction decomposition} that for $1-b\leqslant\sigma\leqslant 2$,
    \[\frac{\zeta'(\sigma+iT)}{\zeta(\sigma+iT)-a}=\sum_{\left|T-\gamma_a\right|\leqslant 1}\frac{1}{\sigma+iT-\rho_a}+\mo\left(\log T\right).\]
    By \eqref{Riemann-von Mangoldt formula}, there are $\mo(\log T)$ terms in the sum above. Moreover, by \eqref{restriction on T for rho_a},
    \[\left|\sigma+iT-\rho_a\right|^{-1}\leqslant \left|T-\gamma_a\right|^{-1}\leqslant\log T.\]
    Hence
    \[\frac{\zeta'(\sigma+iT)}{\zeta(\sigma+iT)-a}\ll\log^2 T\qquad(1-b\leqslant\sigma\leqslant 2).\]
    If $\sigma>2$, $\zeta(\sigma+iT)-a\asymp 1$ and then adopt Lemma \ref{Lemma: convexity bound for zeta} to get
    \[\frac{\zeta'(\sigma+iT)}{\zeta(\sigma+iT)-a}\ll\log^2 T\qquad(\sigma>2).\]
    Using Lemmas \ref{Lemma: convexity bound for zeta} and \ref{Lemma: integral on [2,B]}, we have
    \begin{align*}
    	S^T& \ll \log^2 T\left(\int_{1-b}^0+\int_0^1+\int_1^B\right)\left|\zeta'(\sigma+iT+i\delta)\right|X^{\sigma}\md\sigma\\
    	& \ll\log^4 T\left(\int_{1-b}^0 T^{\frac{1}{2}-\sigma}X^{\sigma}\md\sigma+\int_0^1 T^{\frac{1}{2}(1-\sigma)}X^{\sigma}\md\sigma+\int_1^B X^{\sigma}\md\sigma\right)\\
    	& \ll T^{\frac{1}{2}}\log^4 T.
    \end{align*}
    
    In summary, we have
    \[\sum_{\tau<\gamma_a\leqslant T}\zeta'(\rho_a+i\delta)X^{\rho_a}=S^L-\Delta(X)\frac{T}{2\pi}\sum_{mr=k}c_a(r)m^{-i\delta}\log m+\mo\left(T^{\frac{1}{2}}\log^4 T\right).\]
    
\subsection{Evaluating $S^L$}
    \ 
    
    By Lemma \ref{Lemma: lower bound of zeta}, there is a constant $A=A(a)>\tau$ such that
    \[\left|\frac{a}{\zeta(s)}\right|<\frac{1}{2}\qquad\text{for }s=1-b+it,\ |t|\geqslant A.\]
    Hence the series expansion
    \[\frac{1}{\zeta(s)-a}=\frac{1}{\zeta(s)}\left(1+\frac{a}{\zeta(s)}+\sum_{k=2}^{\infty}\left(\frac{a}{\zeta(s)}\right)^k\right)\]
    is valid on the segment $[1-b+iA,1-b+iT]$. Since
    \[\frac{1}{2\pi i}\int_{1-b+i\tau}^{1-b+iA}\frac{\zeta'(s)}{\zeta(s)-a}\zeta'(s+i\delta)X^s\md s\ll 1,\]
    we have
    \begin{align*}
    	S^L& =\frac{1}{2\pi i}\int_{1-b+iT}^{1-b+iA}\frac{\zeta'}{\zeta}(s)\left(1+\frac{a}{\zeta(s)}+\sum_{k=2}^{\infty}\left(\frac{a}{\zeta(s)}\right)^k\right)\zeta'(s+i\delta)X^s\md s+\mo(1)\\
    	& =:K_1-aK_2+K_3+\mo(1).
    \end{align*}
    To estimate $K_3$, combination of Lemmas \ref{Lemma: convexity bound for zeta} and \ref{Lemma: lower bound of zeta} gives 
    \[K_3\ll\int_A^T t^{b-\frac{1}{2}}\log^5 t\md t\cdot\sum_{k\geqslant 2}\left(\frac{c_0\log T}{\sqrt{T}}\right)^k\ll T^{\frac{1}{2}}\log^7 T.\]
    
    Applying Cauchy's residue theorem, we deduce
    \[K_1=\sum_{A<\gamma\leqslant T}\zeta'(\rho+i\delta)X^{\rho}+\frac{1}{2\pi}\left(\int_{b+iT}^{b+iA}+\int_{1-b+iT}^{b+iT}+\int_{b+iA}^{1-b+iA}\right)\frac{\zeta'}{\zeta}(s)\zeta'(s+i\delta)X^s\md s,\]
    where we have assumed that $\left|T-\gamma\right|\gg\frac{1}{\log T}$. Removing this restriction on $T$ and noticing that $X^s\ll 1$ on the segments $[1-b+iT,b+iT]\cup[b+iA,1-b+iA]$, the analysis in \cite[p. 148]{HuPC} yields
    \[K_1=\sum_{\tau<\gamma\leqslant T}\zeta'(\rho+i\delta)X^{\rho}-\frac{1}{2\pi i}\int_{b+iA}^{b+iT}\frac{\zeta'}{\zeta}(s)\zeta'(s+i\delta)X^s\md s+\mo\left(T^{\frac{1}{2}}\log^3 T\right).\]
    Denote by $K_1^*$ the integral in $K_1$. Then
    \begin{align*}
    	K_1^*& =\frac{X^b}{2\pi}\sum_{m,n\geqslant 1}\frac{\Lambda(m)n^{-i\delta}\log n}{\left(mn\right)^b}\int_A^T e^{it\log\frac{X}{mn}}\md t\\
    	& =\Delta(X)\frac{T}{2\pi}\sum_{mn=X}\Lambda(m)n^{-i\delta}\log n+\mo\left(\log^3 T\right),
    \end{align*}
    where the second line results from integration by parts. Therefore,
    \[K_1=\sum_{\tau<\gamma\leqslant T}\zeta'(\rho+i\delta)X^{\rho}-\Delta(X)\frac{T}{2\pi}\sum_{mn=X}\Lambda(m)n^{-i\delta}\log n+\mo\left(T^{\frac{1}{2}}\log^3 T\right).\]
    
\subsection{Evaluating $K_2$}
    \ 
    
    Using \eqref{functional equation for zeta'/zeta}, we have
    \[\frac{\zeta'(s+i\delta)}{\zeta(s)}=\frac{\chi(s+i\delta)}{\chi(s)}\frac{\zeta(1-s-i\delta)}{\zeta(1-s)}\left(-\log\frac{|t+\delta|}{2\pi}-\frac{\zeta'}{\zeta}(1-s-i\delta)+\mo\bigl(\left|t\right|^{-1}\bigr)\right).\]
    Applying \eqref{chi -- leading behavior}, we get
    \[\frac{\chi(\sigma+i\delta+it)}{\chi(\sigma+it)}=e^{-i\delta\log\frac{t}{2\pi}}\bigl(1+\mo\bigl(\left|t\right|^{-1}\bigr)\bigr).\]
    Then
    \begin{align*}
    	K_2={}& \frac{1}{2\pi i}\int_{1-b+iA}^{1-b+iT}\frac{\zeta'}{\zeta}(s)\frac{\zeta'(s+i\delta)}{\zeta(s)}X^s\md s\\
    	={}& \frac{1}{2\pi i}\int_{1-b+iA}^{1-b+iT}\left(\log\frac{t}{2\pi}+\frac{\zeta'}{\zeta}(1-s)+\mo\bigl(t^{-1}\bigr)\right)\left(\log\frac{t+\delta}{2\pi}+\frac{\zeta'}{\zeta}(1-s-i\delta)+\mo\bigl(t^{-1}\bigr)\right)\\
    	& \qquad\qquad\qquad\ \cdot e^{-i\delta\log\frac{t}{2\pi}}\bigl(1+\mo\bigl(t^{-1}\bigr)\bigr)\frac{\zeta(1-s-i\delta)}{\zeta(1-s)}X^s\md s\\
    	={}& \frac{1}{2\pi i}\int_{1-b+iA}^{1-b+iT}\Bigl(\log\frac{t}{2\pi}+\frac{\zeta'}{\zeta}(1-s)\Bigr)\Bigl(\log\frac{t+\delta}{2\pi}+\frac{\zeta'}{\zeta}(1-s-i\delta)\Bigr)\\
    	& \qquad\qquad\qquad\ \cdot e^{-i\delta\log\frac{t}{2\pi}}\frac{\zeta(1-s-i\delta)}{\zeta(1-s)}X^s\md s+\mo\left(\log^5 T\right)\\
    	={}& \frac{X^{1-b}}{2\pi}\int_A^T\Biggl(\log\frac{t}{2\pi}\log\frac{t+\delta}{2\pi}+\frac{\zeta'}{\zeta}(b-it)\log\frac{t+\delta}{2\pi}+\frac{\zeta'}{\zeta}(b-it-i\delta)\log\frac{t}{2\pi}\\
    	& \qquad\qquad\quad +\frac{\zeta'}{\zeta}(b-it)\frac{\zeta'}{\zeta}(b-it-i\delta)\Biggr)e^{it\log X-i\delta\log\frac{t}{2\pi}}\frac{\zeta(b-it-i\delta)}{\zeta(b-it)}\md s+\mo\left(\log^5 T\right)\\
    	=:{}& I_1+I_2+I_3+I_4,
    \end{align*}
    where we have used the following estimates
    \begin{align*}
    	\frac{\zeta'}{\zeta}(1-s)& =\frac{\zeta'}{\zeta}(b-it)\ll\min\left\{\log^3 t,\log T\right\},\\
    	\frac{\zeta'}{\zeta}(1-s-i\delta)& =\frac{\zeta'}{\zeta}(b-it-i\delta)\ll\min\left\{\log^3 t,\log T\right\},\\
    	\frac{\zeta(1-s-i\delta)}{\zeta(1-s)}& =\frac{\zeta(b-it-i\delta)}{\zeta(b-it)}\ll \log^2 t,\quad X^s\ll 1,
    \end{align*}
    which are obtained from \eqref{Laurent expansion for zeta at 1}, \eqref{lower bound of zeta--1} and Lemma \eqref{Lemma: convexity bound for zeta}.
    
    Note that
    \begin{align*}
    	I_1& =\frac{X^{1-b}}{2\pi}\sum_{m,n\geqslant 1}\frac{\mu(m)n^{i\delta}}{\left(mn\right)^b}\int_A^T e^{it\log(Xmn)-i\delta\log\frac{t}{2\pi}}\log\frac{t}{2\pi}\log\frac{t+\delta}{2\pi}\md t\\
    	& =\frac{X^{1-b}}{2\pi}\Delta\left(X^{-1}\right)\sum_{mn=\frac{1}{X}}\frac{\mu(m)n^{i\delta}}{\left(mn\right)^b}\int_A^T e^{-i\delta\log\frac{t}{2\pi}}\log\frac{t}{2\pi}\log\frac{t+\delta}{2\pi}\md t+\mo\left(\log^4 T\right)\\
    	& =X\Delta\left(X^{-1}\right)\sum_{mn=\frac{1}{X}}\mu(m)n^{i\delta}\int_A^T e^{-i\delta\log\frac{t}{2\pi}}\biggl(\log^2\frac{t}{2\pi}+\mo\bigl(t^{-1}\bigr)\biggr)\md t+\mo\left(\log^4 T\right)\\
    	& =X\Delta\left(X^{-1}\right)\sum_{mn=\frac{1}{X}}\mu(m)n^{i\delta}\cdot\left(\frac{T}{2\pi}\right)^{1-i\delta}\\
    	& \quad \cdot\left(\frac{1}{1-i\delta}\log^2\left(\frac{T}{2\pi}\right)-\frac{2}{\left(1-i\delta\right)^2}\log\frac{T}{2\pi}+\frac{2}{\left(1-i\delta\right)^3}\right)+\mo\left(\log^4 T\right).
    \end{align*}
    Similar, we derive
    \begin{align*}
    	I_2& =X\Delta\left(X^{-1}\right)\sum_{mnr=\frac{1}{X}}\Lambda(m)n^{i\delta}\mu(r)\cdot\left(\frac{T}{2\pi}\right)^{1-i\delta}\left(\frac{1}{1-i\delta}\log\frac{T}{2\pi}-\frac{1}{\left(1-i\delta\right)^2}\right)+\mo\left(\log^3 T\right),\\
    	I_3& =X\Delta\left(X^{-1}\right)\sum_{mnr=\frac{1}{X}}\Lambda(m)\left(mn\right)^{i\delta}\mu(r)\cdot\left(\frac{T}{2\pi}\right)^{1-i\delta}\left(\frac{1}{1-i\delta}\log\frac{T}{2\pi}-\frac{1}{\left(1-i\delta\right)^2}\right)+\mo\left(\log^3 T\right),\\
    	I_4& =X\Delta\left(X^{-1}\right)\sum_{mnr\ell=\frac{1}{X}}\Lambda(m)\left(mn\right)^{i\delta}\mu(r)\Lambda(\ell)\cdot\left(\frac{T}{2\pi}\right)^{1-i\delta}\left(1-i\delta\right)^{-1}+\mo\left(\log^3 T\right).
    \end{align*}
    Combining the expansions above gives an expression for $K_2$.
    
\subsection{Finishing the proof}
    \ 
    
    \eqref{our third result} follows from the estimates obtained in the previous subsections of this section. If, further, $X\in\ZZ_{>0}$, we have
    \[\Delta(X)=1,\ \ \Delta\left(X^{-1}\right)=\mathbbm{1}_{X=1},\ \ \mu(1)=1,\ \ \Lambda(1)=0,\]
    which simplify $K_2$, show \eqref{our third result--X integers} and complete the proof of Theorem \ref{Theorem: our third result}.

\section{\bf Further remarks}\label{Section: further remarks}
    Taking $X=1$ in \eqref{our third result} and differentiating \eqref{our third result} with respect to $\alpha$ yield the asymptotics of
    \[\sum_{1<\gamma_a\leqslant T}\zeta^{(n)}(\rho_a),\quad T\to\infty,\]
    but the consequent expansion is implicit. Indeed, by repeating the proof of Theorem \ref{Theorem: our third result}, we may obtain the explicit expansion in Theorem \ref{Theorem: our last result}, and even get the explicit expansion for
    \[\sum_{1<\gamma_a\leqslant T}\zeta^{(n)}(\rho_a)X^{\rho_a},\quad T\to\infty,\]
    which generalizes Pearce-Crump's result \cite[Corollary 6.1]{PeCr}.
    
    \begin{theorem}\label{Theorem: our last result}
    	Let $a$ be any complex number. Then for $T$ large,
    	\[\sum_{1<\gamma_a\leqslant T}\zeta^{(n)}\left(\rho_a\right)=\sum_{1<\gamma\leqslant T}\zeta^{(n)}\left(\rho\right)+a\sum_{k=0}^{n}\binom{n}{k}\left(-1\right)^{n-k}\bigl\{A_1(k;T)-A_2(k;T)\bigr\}+\mo\left(T^{\frac{1}{2}}\log^{n+6}T\right),\]
    	where the first sum on the right admits the expansion in \cite[Theorem 3]{HuPC},
    	\[A_1(k;T):=\sum_{m\leqslant \frac{T}{2\pi}}\Lambda_k(m)\log^{n-k+1}m,\quad A_2(k;T):=\sum_{mr\leqslant \frac{T}{2\pi}}\Lambda(r)\Lambda_k(m)\log^{n-k}(mr),\]
    	and $\Lambda_k$ is the higher von Mangoldt function defined by \cite[Equation (1.43)]{IwKo}
    	\[\Lambda_k(m):=\sum_{d\mid m}\mu(d)\log^k\Bigl(\frac{m}{d}\Bigr).\]
    \end{theorem}

	\begin{remark}
		To obtain the expansion in $T$, we have to use Perron's formula to establish the full asymptotics of $A_1(k;T)$ and $A_2(k;T)$. Setting $a=0$ in the theorem recovers \cite[Theorem 3]{HuPC}. 
	\end{remark}
	
	The formula \eqref{our third result} also reveals that the function
	\[S_T(a,\delta):=\sum_{\tau<\gamma_a\leqslant T}\zeta'(\rho_a+i\delta)X^{\rho_a}\]
	is analytic in $\delta$. Because $c_0(r)=-\Lambda(r)$ in Theorem \ref{Theorem: our third result}, $S_T(a,\delta)$, as a function in $a$, is analytic at $a=0$ and discontinuous elsewhere. Furthermore, both the factors $\Delta(X)$ and $\Delta(X^{-1})$ appear in \eqref{our third result}. This fact reflects the varied behavior of $S_T(a,\delta)$ in different $X$ ranges, which is more complicated than those described in Gonek \cite[p. 93]{Gone--1} and Pearce-Crump \cite[p. 5]{PeCr}. It would be of interest to explain these phenomenons without calculation.
	
	Finally, as Steuding pointed out in \cite[p. 686]{Steu}, it would be desirable to know any number-theoretical meaning of the coefficients $c_a(r)$ with $a\ne 0$.
	
\section*{\bf Acknowledgements} The authors are grateful to Andrew Pearce-Crump and Kamel Mazhouda for their valuable comments on this paper.


\begin{thebibliography}{99}
	\bibitem{BoLL}
	H. Bohr, E. Landau and J. E. Littlewood, Sur la fonction $\zeta(s)$ dans le voisinage de la droite $\sigma=\frac{1}{2}$, \textit{Bull. Acad. Roy. Belg.} (1913) 3--35.
		
	\bibitem{Fuji--1}
	A. Fujii, On a conjecture of Shanks, \textit{Proc. Japan Acad.} \textbf{70}(4) (1994) 109--114.
	
	\bibitem{Fuji--2}
	A. Fujii, On the distribution of values of the derivative of the Riemann zeta function at its zeros, I, \textit{Proc. Steklov Inst. Math.} \textbf{276} (2012) 51--76.
		
	\bibitem{GaSt}
	R. Garunk\v{s}tis and J. Steuding, On the roots of the equation $\zeta(s)=a$, \textit{Abh. Math. Semin. Univ. Hamb.} \textbf{84}(1) (2014) 1--15.
		
	\bibitem{Gone}
	S. M. Gonek, Mean values of the Riemann zeta-function and its derivatives, \textit{Invent. Math.} \textbf{75} (1984) 123--141.
	
	\bibitem{Gone--1}
	S. M. Gonek, A formula of Landau and mean values of $\zeta(s)$, in \textit{Topics in Analytic Number Theory} (University of Texas Press, Austin, 1985), pp. 92--97.
		
	\bibitem{HuPC}
	C. Hughes and A. Pearce-Crump, A discrete mean-value theorem for the higher derivatives of the Riemann zeta function, \textit{J. Number Theory} \textbf{241} (2022) 142--164.
	
	\bibitem{Ingh}
	A. E. Ingham, Mean-value theorems in the theory of the Riemann Zeta-function, \textit{Proc. Lond. Math. Soc.} \textbf{27} (1927) 273--300.
	
	\bibitem{IwKo}
	H. Iwaniec and E. Kowalski, \textit{Analytic Number Theory}, Colloquium Publications, Vol. 53 (American Mathematical Society, 2004).
	
	\bibitem{JaMa}
	M. T. Jakhlouti and K. Mazhouda, Distribution of the values of the derivative of the Riemann zeta function on its $a$-points, \textit{Unif. Distrib. Theory} \textbf{9}(1) (2014) 115--125.
	
	\bibitem{MoVa}
	H. L. Montgomery and R. C. Vaughan, \textit{Multiplicative Number Theory. I. Classical Theory}, Cambridge Studies in Advanced Mathematics, Vol. 97 (Cambridge University Press, 2006).
		
	\bibitem{NIST-Handbook}
	F. W. J. Olver, A. B. Olde Daalhuis, D. W. Lozier, B. I. Schneider, R. F. Boisvert, C. W. Clark, B. R. Miller, B. V. Saunders, H. S. Cohl and M.A. McClain (Editors), \textit{NIST digital library of mathematical functions}, Release 1.2.2 of 2024-09-15, aviable at \url{https://dlmf.nist.gov/}.
		
	\bibitem{Patt}
	S. J. Patterson, \textit{An Introduction to the Theory of the Riemann Zeta-Function}, Cambridge Studies in Advanced Mathematics, Vol. 14 (Cambridge University Press, 1995).
	
	\bibitem{PeCr}
	A. Pearce-Crump, A further generalization of sums of higher derivatives of the Riemann zeta function, \textit{Int. J. Number Theory} (2025) 1--19.
	
	\bibitem{Steu}
	J. Steuding, One hundred years uniform distribution modulo one and recent applications to Riemann’s zeta-function, in \textit{Topics in mathematical analysis and applications}, Springer Optimization and Its Applications, Vol. 94 (Springer, 2014), pp. 659--698.
	
	\bibitem{Titc}
	E. C. Titchmarsh, \textit{The Theory of the Riemann Zeta-Function}, Oxford Science Publications (Clarendon Press, Oxford, 1986).
		
\end{thebibliography}
\end{document}